\documentclass[11pt,dvips,twoside]{article}

\usepackage{pslatex}
\usepackage{fancyhdr}
\usepackage{graphicx}
\usepackage{geometry}

\RequirePackage{amsfonts,amssymb,amsmath,amscd,amsthm}
\RequirePackage{txfonts}
\RequirePackage{graphicx}
\RequirePackage{xcolor}
\RequirePackage{geometry}
\RequirePackage{enumerate}

\def\figurename{Figure} 
\makeatletter
\renewcommand{\fnum@figure}[1]{\figurename~\thefigure.}
\makeatother

\def\tablename{Table} 
\makeatletter
\renewcommand{\fnum@table}[1]{\tablename~\thetable.}
\makeatother

\usepackage{amsmath}
\usepackage{amssymb}
\usepackage{amsfonts}
\usepackage{amsthm,amscd}

\newtheorem{theorem}{Theorem}[section]
\newtheorem{lemma}[theorem]{Lemma}
\newtheorem{corollary}[theorem]{Corollary}
\newtheorem{proposition}[theorem]{Proposition}
\theoremstyle{definition}
\newtheorem{definition}[theorem]{Definition}
\newtheorem{example}[theorem]{Example}

\theoremstyle{remark}
\newtheorem{remark}[theorem]{Remark}

\numberwithin{equation}{section}

\begin{document}

\title{\bfseries\scshape{Ternary Leibniz color algebras and beyond}}

\author{\bfseries\scshape Ibrahima BAKAYOKO\thanks{e-mail address: ibrahimabakayoko27@gmail.com}\\
D\'epartement de Math\'ematiques,\\
Universit\'e de N'Z\'er\'ekor\'e\\
BP 50, N'Z\'er\'ekor\'e, Guin\'ee.\\
 \\
}
 
\date{}
\maketitle 


\noindent\hrulefill

\noindent {\bf Abstract.} 
The purpose of this paper is to study some ternary color algebras. We generalize some results on ternary Leibniz algebras to the case of ternary
 Leibniz color algebras. In order to produce examples of ternary Leibniz color algebras from Leibniz color algebras, several results on Leibniz 
color algebras are given. Next, we introduce ternary Leibniz color algebras and give some constructions dealing with operators, direct sum and 
tensor product. After that, we introduce bimodule over ternary Leibniz color algebra and point out that the direct sum and the tensor 
product of two bimodules over a ternary Leibniz color algebra is also a bimodule over the given ternary Leibniz color algebra. Then, we introduce
Leibniz-Poisson color algebra and study some of their properties. Finally, we give a brief description of Lie triple systems, Jordan triple 
systems and Comstrans algebras. The connection among them are also studied.

\noindent \hrulefill

\vspace{.3in}


\vspace{.08in} \noindent \textbf{Keywords}: Ternary Leibniz color algebras, color Jordan triple systems, color Lie triple systems, 
ternary Leibniz-Poisson color algebras, associative color trialgebras. 
\vspace{.3in}
\vspace{.2in}
\section{Introduction}
The notion of n-Lie algebras were introduced by Filippov \cite{F} in 1985  as a natural generalization of Lie algebras. 
More precisely, $n$-Lie algebras are vector spaces $V$ equipped with $n$-ary operation which is skew symmetric 
for any pair of variables and satisfies the following identity :
\begin{eqnarray}
 [[x_1, x_2, \dots, x_n], y_1, y_2, \dots, y_{n-2}, y_{n-1}]=\sum_{i=1}^n[x_1, x_2, \dots, x_{i-1}, [x_i, y_1, y_2, \dots, y_{n-2}, y_{n-1}], 
x_{i+1}, \dots, x_n].\label{F}
\end{eqnarray}
For $n=3$, it reads
\begin{eqnarray}
 [[x, y, z], t, u]=[x, y, [z, t, u]]+[x, [y, t, u], z]+[[x, t, u], y, z].\nonumber
\end{eqnarray}
Whenever the identity (\ref{F}) is satisfied and the bracket fails to be totally skew symmetric we  obtain n-Leibniz algebras  \cite{JM}.
Moreover, when the bracket is skew-symmetric with respect to the last two variables, $(V, [-, -, -])$ is said to be quasi-Lie 3-algebras \cite{JMC}.
If in addition, 
\begin{eqnarray}
[x, y, z]+[y, z, x]+[z, x, y]=0\nonumber
\end{eqnarray}
is satisfied for any $x, y, z\in V$, $(V, [-, -, -])$ is called a Lie triple system \cite{JMC}.

The n-Lie algebras found their applications in many fields of mathematics and Physics. For instance, Takhtajan has developed the foundations of 
the theory of Nambu-Poisson manifolds \cite{TL}.
The general cohomology theory for n-Lie algebras and Leibniz n-algebras was established in \cite{RM}.
The structure and classification theory of finite dimensional n-Lie algebras was given by Ling \cite{LW}  and many
other authors. For more details of the theory and applications of n-Lie algebras, see \cite{DI}
and references therein.
 
Specially, 3-Lie algebras are applied to the study of the gauge symmetry and supersymmetry of multiple coincident M2-branes in \cite{BL}.
The authors in \cite{MA} studied non-commutative ternary Nambu-Poisson algebras and their Hom-type version. They provided construction results 
dealing with tensor product and direct sums of two (non-commutative) ternary (Hom-) Nambu-Poisson algebras. Examples and a 3-dimensional
 classification of non-commutative ternary Nambu-Poisson algebras were given.

The concept of n-Lie algebras are extended to the graded case by  Zhang T. in \cite{TZ}, in which he studied the cohomology and deformations of
 n-Lie colour algebras when $n=3$, as well as the abelian extensions of 3-Lie colour algebras. 
Ivan Kaygorodova and  Yury Popov studied genera-lized derivations of n-ary color algebras \cite{IY}.

It well-known that mathematical objects are often understood through studying operators defined on them. For instance, in Gallois theory a field is studied
by its automorphisms, in analysis functions are studied through their derivations,  and in geometry manifolds are studied through their vector fields., 
Fifty years ago, several operators have been found from studies in analysis, probability and combinatorics. Among these operators, one can cite 
averaging operator, Reynolds operator, Leroux's TD operator, Nijenhuis operator and Rota-Baxter operator.

The Rota-Baxter operator originated from the work of G. Baxter \cite{GB} on Spitzer's identity\cite{F} in fluctuation theory.
For example, on the polynomial algebra, the indefinite integral
$$R(f)(x)=\int_0^xf(t)dt$$
 and the inverse of any bijective derivation are Rota-Baxter operators.\\
Rota- Baxter algebras are used in many fiels of mathematics and mathematical Physics. 
In mathematics, they used in algebra,
 number theory, operads and combinatorics \cite{MA}, \cite{CA1}, \cite{CA2}, \cite{PC}, \cite{GCB}. In 
mathematical physics they appear as the operator form of the classical Yang Baxter equation \cite{} or as the fondamental algebraic strucutre in
 the normalisation of quantum fild theory of Connes and Kreimer \cite{CK}.

The Nijenhuis operator on an associative algebra was introduced in \cite{CJ} to study quantum bi-Hamiltonian
systems while the notion Nijenhuis operator on a Lie algebra originated from the concept of Nijenhuis tensor that was introduced by Nijenhuis
in the study of pseudo-complex manifolds and was related to the well known concepts of Schouten-Nijenhuis bracket , the Frolicher-Nijenhuis
 bracket \cite{FN}, and the Nijenhuis-Richardson bracket. The associative analog of the Nijenhuis relation may be regaded as the homogeneous version of 
Rota-Baxter relation\cite{PL}.

In non-associative algebra, the Rota-Baxter operators are used  in order to produce another one of the same type or not from the previous one.

The aim of this paper is to study some structures of ternary  color algebras. The paper is organized as follows.
 In section two, we recall basic notions concerning, graded vector spaces, bicharacter and associative color algebras
In section three, we introduce Leibniz color algebras and Leibniz-Poisson color algebra. We investigate some properties of Leibniz color algebras
and give some constructions  dealing with averaging operator, element of centroid Nijenhuis operator, Rota-Baxter operator and 
Reynolds. We introduce action of Leibniz color algebra onto another another one and define the semidirect sum of Leibniz color algebras.
 In section four,  we introduce
 ternary Leibniz color algebras and give some constructions from Leibniz color algebra, element of centroid, averaging operator, Rota-Baxter 
operator and Reynolds operator. We also prove that the tensor product of two ternary Leibniz color algebra is a Leibniz color algebra.
 Moreover, we show that the tensor product of a commutative associative color algebra and  ternary Leibniz color algebra is also a ternary
 Leibniz color algebra.

%
%
%
%
%
%
%
%

Throughout this paper, all graded vector spaces are assumed to be over a field $\mathbb{K}$ of characteristic different from 2.

\pagestyle{fancy} \fancyhead{} \fancyhead[EC]{Ibrahima Bakayoko} 
\fancyhead[EL,OR]{\thepage} \fancyhead[OC]{Ternary Leibniz color algebras and beyond} \fancyfoot{}
\renewcommand\headrulewidth{0.5pt}

\section{Preliminaries}
In this section, we give the definitions of associative color algebras, Lie color  algebras,
 averaging operators on associative and Lie color algebras, and constructions of  Leibniz color algebras.
\begin{definition}
 \begin{enumerate}
  \item [1)] Let $G$ be an abelian group. A vector space $V$ is said to be a $G$-graded if, there exists a family $(V_a)_{a\in G}$ of vector 
subspaces of $V$ such that
$$V=\bigoplus_{a\in G} V_a.$$
\item [2)] An element $x\in V$ is said to be homogeneous of degree $a\in G$ if $x\in V_a$. We denote $\mathcal{H}(V)$ the set of all homogeneous elements
in $V$.
\item [3)] Let $V=\oplus_{a\in G} V_a$ and $V'=\oplus_{a\in G} V'_a$ be two $G$-graded vector spaces. A linear map $f : V\rightarrow V'$ is said 
to be homogeneous of degree $b$ if 
$$f(V_a)\subseteq  V'_{a+b}, \forall a\in G.$$
If, $f$ is homogeneous of degree zero i.e. $f(V_a)\subseteq V'_{a}$ holds for any $a\in G$, then $f$ is said to be even.
 \end{enumerate}
\end{definition}
\begin{definition}
  \begin{enumerate}
   \item [1)] An algebra $(A, \cdot)$ is said to be $G$-graded if its underlying vector space is $G$-graded i.e. $A=\bigoplus_{a\in G}A_a$, and if furthermore 
$A_a\cdot A_b\subseteq A_{a+b}$, for all $a, b\in G$.

 Let $A'$ be another $G$-graded algebra.
\item  [2)] A morphism $f : A\rightarrow A'$ 
of $G$-graded algebras
is by definition an algebra morphism from $A$ to $A'$ which is, in addition an even map.
  \end{enumerate}
\end{definition}

\begin{definition}
 Let $G$ be an abelian group. A map $\varepsilon :G\times G\rightarrow {\bf \mathbb{K}^*}$ is called a skew-symmetric bicharacter on $G$ if the following
identities hold, 
\begin{enumerate} 
 \item [(i)] $\varepsilon(a, b)\varepsilon(b, a)=1$,
\item [(ii)] $\varepsilon(a, b+c)=\varepsilon(a, b)\varepsilon(a, c)$,
\item [(iii)]$\varepsilon(a+b, c)=\varepsilon(a, c)\varepsilon(b, c)$,
\end{enumerate}
$a, b, c\in G$,
\end{definition}

\begin{example} Some standard examples of skew-symmetric bicharacters are:
 \begin{enumerate}
\item [1)] $G=\mathbb{Z}_2,\quad \varepsilon(i, j)=(-1)^{ij}$, or more generally
\begin{gather*} G=\mathbb{Z}_2^n=\{(\alpha_1, \dots, \alpha_n)| \alpha_i\in\mathbb{Z}_2 \}, \\
\varepsilon((\alpha_1, \dots, \alpha_n), (\beta_1, \dots, \beta_n)):= (-1)^{\alpha_1\beta_1+\dots+\alpha_n\beta_n}.
\end{gather*}
\item [2)] $G=\mathbb{Z}_2\times\mathbb{Z}_2,\quad \varepsilon((i_1, i_2), (j_1, j_2)=(-1)^{i_1j_2-i_2j_1}$,
\item [3)] $G=\mathbb{Z}\times\mathbb{Z} ,\quad \varepsilon((i_1, i_2), (j_1, j_2))=(-1)^{(i_1+i_2)(j_1+j_2)}$,
\item [4)] $G=\{-1, +1\} , \quad\varepsilon(i, j)=(-1)^{(i-1)(j-1)/{4}}$.
\end{enumerate}
\end{example}

If x and y are two homogeneous elements of degree $a$ and $b$ respectively and $\varepsilon$ is a skew-symmetric bicharacter, 
then we shorten the notation by writing $\varepsilon(x, y)$ instead of $\varepsilon(a, b)$.

\begin{definition}
A color algebra is a $G$-graded algebra $(A, \cdot)$ together with a bicharacter  $\varepsilon$.
\end{definition}

\begin{definition}
An associative color algebra is a $G$-graded algebra $(A, \cdot)$ together with a bicharacter 
$\varepsilon :G\times G\rightarrow {\bf \mathbb{K}^*}$ such that 
\begin{eqnarray}
(x\cdot y)\cdot z &=&x\cdot(y\cdot z) \qquad\qquad(\mbox{associativity})
 \end{eqnarray}
for all $x, y, z\in\mathcal{H}(A)$.
\end{definition}

\section{Leibniz color algebras}
One of the main result of this paper is based on the fact that one may associate a ternary Leibniz color algebra to Leibniz color algebra. To this
end and to give various examples of this contruction, we develope, in this section, several properties and constructions of Leibniz color algebras.
\subsection{Generalities}

In this subsection, we give some results on Leibniz color algebras. They will provide examples of
 ternary Leibniz color algebras from other algebraic structures.
\begin{definition}\cite{BD} 
 A Leibniz color algebra is a $G$-graded vector space $L$ together with an even bilinear map 
$[-, -] : L\otimes L\rightarrow L$ and a bicharacter $\varepsilon : G\otimes G\rightarrow \mathbb{K}^*$  such that
 \begin{eqnarray}
   [[x, y], z]=[x, [y, z]]+\varepsilon(y, z)[[x, z], y] \label{cpa}
 \end{eqnarray}
holds, for all $x, y, z\in \mathcal{H}(L)$.
\end{definition}

 \begin{example}
 Let $A=A_0\oplus A_1=<e_1, e_2>\oplus<e_3>$ be a two-dimensional superspace. The multiplications
\begin{eqnarray}
 [e_1, e_1]= [e_1, e_2]=[e_2, e_1]=0 \,\,\mbox{and}\,\,  [e_2, e_2]=e_1, \nonumber
\end{eqnarray}
make $L$ into a Leibniz  superalgebra, for any $a, b\in\mathbb{R}$.
 \end{example}

\begin{lemma}\label{rr}
 Let $(L, [-, -], \varepsilon)$ be a Leibniz color algebra. Then
$$R_{[x, x]}=0,$$
for any $x\in L$.
\end{lemma}
\begin{proof}
 Taking $y=z$ in (\ref{cpa}), we have for any $x, y\in L$,
\begin{eqnarray}
 [x, [y, y]]=[[x, y], y]-\varepsilon(y, y)[[x, y], y]=0,
\end{eqnarray}
and the conclusion follows.
\end{proof}

\begin{definition}\label{c}
 Let $(L, [-, -], \varepsilon)$ be a Leibniz color algebra. Then the subset 
\begin{eqnarray}
 C_l(L):=\{c\in L|\,\, L_c=0\}=\{c\in L|\,\, [c, x]=0, \forall x\in L\}
\end{eqnarray}
is called the left center of $L$.
\begin{eqnarray}
 C_r(L):=\{c\in L|\,\, R_c=0\}=\{c\in L|\,\, [x, c]=0, \forall x\in L\}
\end{eqnarray}
is called the right center of $L$.
\begin{eqnarray}
 C(L)=C_l(L)\cap C_r(L)
\end{eqnarray}
is called the center of $L$.
\end{definition}

\begin{proposition}\label{lb}
 Let $L$ be a Leibniz color algebra.  Then 
$$[C_r(L), L]\subseteq C_r(L)\quad\mbox{and}\quad [L, C_r(L)]=0.$$
In particular, $C_r(L)$ is an abelian ideal of $L$.
\end{proposition}
\begin{proof}
 By (\ref{cpa}), for any $x, y\in \mathcal{H}(L), c\in \mathcal{H}(C_r(L))$, 
\begin{eqnarray}
 [x, [c, y]]
&=&[[x, c], y]-\varepsilon(c, y)[[x, y], c]\nonumber\\
&=&[R_c(x), y]-\varepsilon(c, y)R_c([x, y])\nonumber\\
&=&0.\nonumber
\end{eqnarray}
Which means that $[c, y]\in C_r(L)$. The second assertion comes from definition \ref{c}.
\end{proof}

\begin{definition}
 Let $L$ be a Leibniz color algebra. The Leibniz kernel of $L$ is defined as 
\begin{eqnarray}
 \mbox{Leib}(L):=\{[x, x],\,\,x\in L\}
\end{eqnarray}
\end{definition}

\begin{remark}
 The Leibniz kernel measures how much a Leibniz algebra deviates from being a Lie algebra. In particular, a Leibniz color algebra is a 
Lie color algebra if and only if its Leibniz kernel vanishes.
\end{remark}

\begin{proposition}
  Let $L$ be a Leibniz algebra. Then
\begin{enumerate}
 \item [i)]
 $[\mbox{Leib}(L), L]\subseteq \mbox{Leib}(L)\quad\mbox{and}\quad \mbox{Leib}(L)\subseteq C_r(L).$
\item [ii)]
$\mbox{Leib}(L)$ is an abelian color ideal of $L$. Moreover, if $L\neq 0$, then $\mbox{Leib}(L)\neq L$.
\end{enumerate}
\end{proposition}
\begin{proof}
i) For any $[x, x]\in \mbox{Leib}(L), y\in L$,
\begin{eqnarray}
 [[x, x]+y, [x, x]+y]
&=&[[x, x], [x, x]]+[[x, x], y]+[y, [x, x]]+[y, y]\nonumber\\
&=&[[x, x], y]+[y, y] \quad\mbox{(By Lemma}\;\; \ref{rr})\nonumber.
\end{eqnarray}
That is $[[x, x], y]=[[x, x]+y, [x, x]+y]-[y, y]\in \mbox{Leib}(L).$\\ The second statement follows from Lemma \ref{rr}.\\
ii) Lemma \ref{rr} and Proposition \ref{lb} b) mean that $\mbox{Leib}(L)$ is a color ideal of $L$. 
Moreover, we have $[\mbox{Leib}(L), \mbox{Leib}(L)]=0$. For the second part, suppose that $\mbox{Leib}(L)=L$. Then, $[L, L]=0$. In particular,
every square of $L$ is zero. Therefore, $L=\mbox{Leib}(L)=0$.
\end{proof}

%

%

\subsection{Constructions}
In this subsection, we recall definitions of special even linear operator and give constructions using these maps.

\begin{definition}\label{bk3}
Let $(A, \cdot, \varepsilon)$ be a color algebra. Then, an even linear map $\varphi : A\rightarrow A$ is said to be :
\begin{enumerate}
 \item [i)] An averaging operator if
\begin{eqnarray}
 \varphi(\varphi(x)\cdot y)=\varphi(x)\cdot \varphi(y)=\varphi(x\cdot \varphi(y)),
\end{eqnarray}
\item [ii)] An element of centroid
\begin{eqnarray}
 \varphi(x\cdot y)=\varphi(x)\cdot y=x\cdot \varphi(y),
\end{eqnarray}
\end{enumerate}
for all $x, y\in\mathcal{H}(A)$.
\end{definition}

\begin{proposition}\label{ed}
 Let $(L, [-, -], \varepsilon)$ be a Leibniz color algebra and $\alpha: L\rightarrow L$ be an injective averaging operator. Then
$L$ is also Leibniz color algebra with respect to the bracket
$$[x, y]_\alpha=[\alpha(x), y]$$
for all $x, y\in\mathcal{H}(L)$.
\end{proposition}
\begin{proof}
For any $x, y, z\in\mathcal{H}(L)$,
  \begin{eqnarray}
  [[x, y]_\alpha, z]_\alpha
&=&[\alpha[\alpha(x), y], z]\nonumber\\
&=&[[\alpha(x), \alpha(y)], z]\nonumber\\
&=&[\alpha(x), [\alpha(y), z]]+\varepsilon(y, z)[[\alpha(x), z], \alpha(y)]\nonumber.
 \end{eqnarray}
As
\begin{eqnarray}
 \alpha([[\alpha(x), z], \alpha(y)])
=[\alpha([\alpha(x), z]), \alpha(y)]
= \alpha([\alpha([\alpha(x), z]), y])
=\alpha([[x, z]_\alpha, y]_\alpha)\nonumber.
\end{eqnarray}
It follows that, 
 \begin{eqnarray}
 &&\qquad \alpha\Big([[x, y]_\alpha, z]_\alpha)
-[\alpha(x), [\alpha(y), z]])-\varepsilon(y, z)[[\alpha(x), z], \alpha(y)]\Big)\nonumber\\
&&=\alpha\Big([[x, y]_\alpha, z]_\alpha)
-[x, [y, z]_\alpha]_\alpha)-\varepsilon(y, z)[[x, z]_\alpha, y]_\alpha\Big)\nonumber\\
&&=0.
 \end{eqnarray}
The conclusion comes from injectivity.
\end{proof}

\begin{proposition}\label{ed}
 Let $(L, [-, -], \varepsilon)$ be a Leibniz color algebra and $\eta: L\rightarrow L$ an element of centroid. Then
$L$ is also Leibniz color algebra with respect to the bracket
$$[x, y]_\eta=[\eta(x), y]$$
for all $x, y\in\mathcal{H}(L)$.
\end{proposition}
\begin{proof}
 For any $x, y, z\in\mathcal{H}(L)$,
  \begin{eqnarray}
  [[x, y]_\eta, z]_\eta
&=&[\eta[\eta(x), y], z]\nonumber\\
&=&[[\eta(x), \eta(y)], z]\nonumber\\
&=&[\eta(x), [\eta(y), z]]+\varepsilon(y, z)[[\eta(x), z], \eta(y)]\nonumber\\
&=&[\eta(x), [\eta(y), z]]+\varepsilon(y, z)[\eta[\eta(x), z], y]\nonumber\\
&=&[x, [y, z]_\eta]_\eta)+\varepsilon(y, z)[[x, z]_\eta, y]_\eta.\nonumber
 \end{eqnarray}
This ends the proof.
\end{proof}

To continue to give other construction from special linear maps we give the below definition.
\begin{definition}\label{bk3}
Let $(A, \cdot, \varepsilon)$ be a color algebra. Then, an even linear map $\varphi : A\rightarrow A$ is said to be
\begin{enumerate}
\item [i)] A Nijenhuis operator if 
\begin{eqnarray}
 \varphi(x)\cdot \varphi(y) = \varphi\Big(\varphi(x)\cdot y + x\cdot \varphi(y) -\varphi(x\cdot y)\Big), \label{nj1}
\end{eqnarray}
\item [ii)] A Reynolds operator if 
\begin{eqnarray}
 \varphi(x)\cdot \varphi(y) = \varphi\Big(\varphi(x)\cdot y + x\cdot \varphi(y) -\varphi(x)\cdot\varphi(y)\Big),
\end{eqnarray}
\item [iii)] A Rota-Baxter operator (of weight $\lambda\in\mathbb{K}$) if
\begin{eqnarray}
 \varphi(x)\cdot \varphi(y) = \varphi\Big(\varphi(x)\cdot y + x\cdot \varphi(y) +\lambda x\cdot y\Big),
\end{eqnarray}
\end{enumerate}
for all $x, y\in\mathcal{H}(A)$.
\end{definition}

\begin{proposition}\label{ed}
 Let $(L, [-, -], \varepsilon)$ be a Leibniz color algebra and $P: L\rightarrow L$ a Reynolds operator. Then
$L$ is also Leibniz color algebra with respect to the bracket
$$[x, y]_P=[P(x), y]+[x, P(y)]-[P(x), P(y)]$$
for all $x, y\in\mathcal{H}(L)$.
Moreover, $P$ is a morphism of  $(L, [-, -]_P, \varepsilon)$ onto  $(L, [-, -], \varepsilon)$.
\end{proposition}
\begin{proof}
 For any $x, y\in\mathcal{H}(L)$,
 \begin{eqnarray}
  [[x, y]_P, z]_P
&=&[[P(x), y]+[x, P(y)]-[P(x), P(y)], z]_P\nonumber\\
&=&[[P(x), P(y)], z]+[[P(x), y], P(z)]+[[x, P(y)], P(z)]\nonumber\\
&&-[[P(x), P(y)], P(z)]-[[P(x), P(y)], P(z)]\nonumber\\
&=&[P(x), [P(y), z]]+[P(x), [y, P(z)]+[x, [P(y), P(z)]]-[P(x), [P(y), P(z)]]\nonumber\\
&&-[P(x), [P(y), P(z)]]+\varepsilon(y, z)\Big([[P(x), z], P(y)]+[[P(x), P(z)], y]+[[x, [P(z)], P(y)]\nonumber\\
&&-[[P(x), P(z)], P(y)]-[[P(x), P(z)], P(y)]\Big)\nonumber\\
&=&[P(x), [P(y), z]+[y, P(z)]-[P(y), P(z)]]+[x, [P(y), P(z)]]-[P(x), [P(y), P(z)]]\nonumber\\
&&+\varepsilon(y, z)\Big([[P(x), P(z)], y]+[[P(x), z]+[x, P(z)]-[P(x), P(z)], P(y)]\nonumber\\
&&-[[P(x), P(z)], P(y)]\Big)\nonumber\\
&=&[P(x), [y, z]_P]+[x, P([y, z]_P)]-[P(x), P([y, z]_P)]\nonumber\\
&&+\varepsilon(y, z)\Big([P([x, z]_P), y]+[[x, z]_P, P(y)]-[P([x, z]_P), P(y)]\Big)\nonumber\\
&=&[x, [y, z]_P]_P+\varepsilon(y, z)[[x, z]_P, y]_P\nonumber.
 \end{eqnarray}
This completes the proof.
\end{proof}

\begin{proposition}\label{ed}
 Let $(L, [-, -], \varepsilon)$ be a Leibniz color algebra and $R: L\rightarrow L$ a Rota-Baxter operator of weight $\lambda$. Then
$L$ is also Leibniz color algebra  with respect to the bracket
$$[x, y]_R=[R(x), y]+[x, R(y)]+\lambda[x, y]$$
for all $x, y\in\mathcal{H}(L)$.
Moreover, $R$ is a morphism of  $(L, [-, -]_R, \varepsilon)$ onto  $(L, [-, -], \varepsilon)$.
\end{proposition}
\begin{proof}
The proof is similar to the next one.
 \end{proof}

\begin{proposition}\label{ed}
 Let $(L, [-, -], \varepsilon)$ be a Leibniz color algebra and $N: L\rightarrow L$ a Nijenhuis operator. Then
$L$ is also Leibniz color algebra  with respect to the bracket
$$[x, y]_N=[N(x), y]+[x, N(y)]-N([x, y])$$
for all $x, y\in\mathcal{H}(L)$.
Moreover, $N$ is a morphism of  $(L, [-, -]_N, \varepsilon)$ onto  $(L, [-, -], \varepsilon)$.
\end{proposition}
\begin{proof}
For any $x, y\in\mathcal{H}(L)$,
 \begin{eqnarray}
  [[x, y]_N, z]_N
&=&[N[x, y]_N, z]+ [[x, y]_N, N(z)]-N[[x, y]_N, z]\nonumber\\
&=&[[N(x), N(y)], z]+[[N(x), y], N(z)]+[[x, N(y)], N(z)]-[N[x, y], N(z)]\nonumber\\
&&-N[[N(x), y], z]-N[[x, N(y)], z]+N[N[x, y], z]\nonumber\\
&=&[[N(x), N(y)], z]+[[N(x), y], N(z)]+[[x, N(y)], N(z)]-N[N[x, y], z]-N[[x, y], N(z)]\nonumber\\
&&+N^2[[x, y], z]-N[[N(x), y], z]-N[[x, N(y)], z]+N[N[x, y], z]\nonumber.
 \end{eqnarray}
By exchanging the role of $y$ and $z$, we get
\begin{eqnarray}
  [[x, z]_N, y]_N
&=&[[N(x), N(z)], y]+[[N(x), z], N(y)]+[[x, N(z)], N(y)]-N[N[x, z], y]-N[[x, z], N(y)]\nonumber\\
&&+N^2[[x, z], y]-N[[N(x), z], y]-N[[x, N(z)], y]+N[N[x, z], y]\nonumber.
 \end{eqnarray}
Then,
 \begin{eqnarray}
  [x, [y, z]_N]_N
&=&[N(x), [y, z]_N]+[x, N[y, z]_N]-N[x, [y, z]_N]\nonumber\\
&=&[N(x), [N(y), z]]+[N(x), [y, N(z)]]-[N(x), N[y, z]]+[x, [N(y), N(z)]]\nonumber\\
&&-N[x, [N(y), z]]-N[x, [y, N(z)]]+N[x, N[y, z]]\nonumber\\
&=&[N(x), [N(y), z]]+[N(x), [y, N(z)]]-N[N(x), [y, z]]-N[x, N[y, z]]+N^2[x, [y, z]]\nonumber\\
&&+[x, [N(y), N(z)]]-N[x, [N(y), z]]-N[x, [y, N(z)]]+N[x, N[y, z]]\nonumber.
 \end{eqnarray}
We can observe, by using Leibniz rule, that the right hand side of the firt identity is equal to the sum of the right hand side of the third
identity plus the right hand side of the second identity multiplied by $\varepsilon(y, z)$.
\end{proof}

Now, we introduce action of Leibniz color algebra on another one.

\begin{definition}\label{le}
Let $L$ and $\mathcal{L}$ be two Leibniz color algebras. A color action of $\mathcal{L}$ on $L$ consists of a pair of bilinear maps, 
$L\times \mathcal{L}\rightarrow {L}, (x, a)\mapsto [x, a]$ and 
$\mathcal{L}\times L\rightarrow {L}, (a, x)\mapsto [a, x]$, such that
\begin{eqnarray}
\left[\left[x, a\right], b\right]&=&\left[x,\left[a, b\right]\right]+\varepsilon(a, b)\left[\left[x, b\right], a\right]\label{la1}\\
\left[\left[a, x\right], b\right]&=&\left[a,\left[x, b\right]\right]+\varepsilon(x, b)\left[\left[a, b\right], x\right]\\
\left[\left[a, b\right], x\right]&=&\left[a,\left[b, x\right]\right]+\varepsilon(b, x)\left[\left[a, x\right], b\right]\label{la3}\\
\left[\left[a, x\right], y\right]&=&\left[a,\left[x, y\right]\right]+\varepsilon(x, y)\left[\left[a, y\right], x\right]\label{la4}\\
\left[\left[x, a\right], y\right]&=&\left[x,\left[a, y\right]\right]+\varepsilon(a, y)\left[\left[x, y\right], a\right]\\
\left[\left[x, y\right], a\right]&=&\left[x,\left[y, a\right]\right]+\varepsilon(y, a)\left[\left[x, a\right], y\right]\label{la6}
\end{eqnarray}
for all $x, y\in L, a, b\in \mathcal{L}$.
\end{definition}

\begin{remark}
i) Whenever, $\mathcal{L}$ is just a color vector space (i.e. has not the structure of Leibniz color algebra), axioms (\ref{la1})-(\ref{la3})
disappear, and axioms (\ref{la4})-(\ref{la6}) mean that $\mathcal{L}$ is a bimodule over $L$.
ii)  Any Leibniz color algebra or any Leibniz superalgebras \cite{CQZ} is a bimodule over itself.
\end{remark}

\begin{proposition}
Let $L$ and $\mathcal{L}$ be two Leibniz color algebras.
Given a Leibniz color  action of $\mathcal{L}$ on $L$, we can consider the  {\it semidirect sum} Leibniz color algebra $L\rtimes \mathcal{L}$, 
which consists of color vector space $L\oplus \mathcal{L}$
\begin{eqnarray}
 [(x, a), (y, b)]=([x, y], [a, b]+[x, b]+[a, y]),
\end{eqnarray}
for all $(x, a), (y, b)\in \mathcal{H}(L\times \mathcal{L})$.
\end{proposition}
\begin{proof}
It uses axioms in Definition \ref{le}
\end{proof}

\begin{corollary}
 Let $M$ be a color bimodule over a Leibniz color algebra $L$. Then $L\times M$ is a Leibniz color algebra with the multiplication
\begin{eqnarray}
 [(x, a), (y, b)]=([x, y], [x, b]+[a, y]),
\end{eqnarray}
for all $(x, a), (y, b)\in \mathcal{H}(L\times \mathcal{L})$.
\end{corollary}

\subsection{Associative color trialgebras}
We now introduce  color trialgebras and establish their relationship with ternary Leibniz color algebras.
\begin{definition}
 An associative color trialgebra is a $G$-graded vector space $A$ equipped with a bicharacter $\varepsilon : G\otimes G\rightarrow \mathbb{K}^*$
and three even binary associative operations $\dashv, \perp, \vdash :
A\otimes A\rightarrow A$ (called left, middle and right respectively), satisfying the following relations :
\begin{eqnarray}
 (x\dashv y)\dashv z&=&x\dashv(y\vdash z)=x\dashv(y\perp z)\\
(x\vdash y)\dashv z&=&x\vdash(y\dashv z)\\
(x\dashv y)\vdash z&=&x\vdash(y\vdash z)=(x\perp y)\vdash z\\
(x\perp y)\dashv z&=&x\perp(y\dashv z)\\
(x\dashv y)\perp z&=&x\perp(y\vdash z)\\
(x\vdash y)\perp z&=&x\vdash(y\perp z)
\end{eqnarray}
for  all $x, y, z\in \mathcal{H}(A)$.
\end{definition}
\begin{example}
 Any associative color algebra $(A, \cdot, \varepsilon)$ is a color trialgebra with $\cdot=\dashv=\perp=\vdash$.
\end{example}
\begin{example}
 Any associative color dialgebra is a color trialgebra with trivial middle product.
\end{example}

\begin{example}
 If $(A, \dashv, \perp, \vdash, \varepsilon)$ is a color trialgebra, then so is $(A, \dashv', \perp', \vdash', \varepsilon)$, where
$$x\dashv' y:= y\vdash x, \quad x\perp' y:= y\perp x, \quad x\vdash' y:= y\dashv x.$$
\end{example}
The following proposition connects color trialgebras to Leibniz-Poisson color algebras.
It will gives a construction of  ternary Leibniz-Poisson color
algebras form  Leibniz-Poisson color algebras.
\begin{proposition}\label{tria}
 Let $(A, \dashv, \perp, \vdash, \varepsilon)$ be a color trialgebra. Then $(A, \cdot, [-, -], \varepsilon)$ is a
 Leibniz-Poisson color algebra with respect to the operations
$$ x\cdot y:= x\perp y \quad [x, y]=x\dashv y-\varepsilon(x, y)x\vdash y, $$
for  all $x, y, z\in \mathcal{H}(A)$.
\end{proposition}

 We end this subsection by establishing a connection between trialgebras and $-1$-tridendriform algebras (with trivial grading).
\begin{theorem}
 Let $(A, \dashv, \perp, \vdash)$ be an associative trialgebra. Then $(A, \dashv, \vdash, \perp)$ is a $-1$-tridendriform algebra.
\end{theorem}
\begin{proof}
 The proof comes from both definitions.
\end{proof}

\begin{corollary}\label{ta}
Let $(A, \dashv, \perp, \vdash)$ be an associative trialgebra. Then $A$ is an associative algebra with respect to the multiplication
$\ast : A\otimes A\rightarrow A$ :
$$x\ast y=x\dashv y+x\vdash y-x\perp y$$
for any $x, y\in A$.
\end{corollary}
\begin{corollary}
 Let $(A, \dashv, \perp, \vdash)$ be an associative trialgebra. Then $(A, \ast, [-, -])$ is a  Leibniz-Poisson algebra with 
$$x\ast y=x\dashv y+x\vdash y-x\perp y\quad{and}\quad [x, y]=x\ast y-y\ast x$$
for any $x, y\in A$.
\end{corollary}

The below corollaries are based on some results of \cite{MD} and the following remark.
\begin{remark}
 If $R$ is a Rota-Baxter operator of weight $\lambda\in\mathbb{K}$ on the trialgebra $(A, \dashv, \perp, \vdash)$, it is also a Rota-Baxter 
operator of weight $\lambda\in\mathbb{K}$ on the associative algebra $(A, \ast)$ of Corollary \ref{ta}.
\end{remark}

\begin{corollary}
 Let $(A, \dashv, \perp, \vdash, R)$ be a Rota-Baxter trialgebra of weight $0$. Then $(A, \star)$ is a left-symmetric algebra with 
$$x\star y=R(x)\ast y-y\ast R(x)\quad{and}\quad  x\ast y=x\dashv y+x\vdash y-x\perp y$$
for all $x, y\in A$.
\end{corollary}

\begin{corollary}
Let $(A, \dashv, \perp, \vdash, R)$ be a Rota-Baxter trialgebra of weight $-1$. Then 
$(A, \star)$ is an associative algebra with 
$$x\star y=R(x)\ast y-y\ast R(x)-x\ast y\quad{and}\quad  x\ast y=x\dashv y+x\vdash y-x\perp y.$$ 
\end{corollary}

Now, let us recall the definition of Leibniz-Poisson color algebras.

\begin{definition} 
 A  Leibniz-Poisson color algebra is a $G$-graded vector space $P$ together with two even bilinear maps 
$[-, -] : P\otimes P\rightarrow P$ and $\cdot : P\otimes P\rightarrow P$ and a bicharacter $\varepsilon : G\otimes G\rightarrow \mathbb{K}^*$
  such that
\begin{enumerate}
\item [1)] $(P, \cdot, \varepsilon)$ is an associative color algebra,
 \item [2)]  $(P, [-, -], \varepsilon)$ is a Leibniz color algebra,
\item [3)] and the following {\it right Leibniz} identity :
 \begin{eqnarray}
  [x\cdot y, z]=x\cdot [y, z]+\varepsilon(y, z) [x, z]\cdot y \label{comp}
 \end{eqnarray}
holds, for all $x, y, z\in \mathcal{H}(P)$.
\end{enumerate}
\end{definition}

\begin{remark}
1) When the color associative product $\cdot$ is $\varepsilon$-commutative i.e. $x\cdot y=\varepsilon(x, y)y\cdot x$, 
then $(P, \cdot, [-, -], \varepsilon)$ is said to be
a commutative Leibniz-Poisson color algebra.\\
2) A non-commutative Leibniz-Poisson color algebra  in which the associative product is non-commutative and the bracket $[-, -]$ is  
$\varepsilon$-skew symmetric, is called a non-commutative Poisson color algebra  \cite{BD}.\\
3) Whenever the color associative product $\cdot$ is $\varepsilon$-commutative and the bracket $[-, -]$ is $\varepsilon$-skew-symmetric,
then $(P, \cdot, [-, -], \varepsilon)$ is named a commutative Poisson color algebra.
\end{remark}

\begin{example}\label{laa1}(\cite{BD})
Let $(D, \dashv, \vdash, \varepsilon)$ be an associative color dialgebra. Then $(D, \dashv, [-, -], \varepsilon)$ is a non-commutative
 Leibniz-Poisson color algebra with the bracket
\begin{eqnarray}
 [x, y]:=x\dashv y-\varepsilon(x, y) y\vdash x,\nonumber
\end{eqnarray}
for all $x, y\in \mathcal{H}(D)$.
\end{example}

\section{Ternary color algebras}
This section is devoted to the construction of some structures of ternary  color algebras.
\subsection{Ternary Leibniz color algebras}
\begin{definition}
 A ternary Leibniz color algebra is a $G$-graded vector space $A$ over a field $\mathbb{K}$ equipped with a bicharacter 
$\varepsilon : G\otimes G\rightarrow \mathbb{K}^*$ and an even trilinear operation $[-, -, -] : A\otimes A\otimes A\rightarrow A$ (i.e.
$[x, y, z]\subseteq  A_{x+y+z}$ whenever  $x, y, z\in \mathcal{H}(A)$) satisfying the following :
\begin{eqnarray}
 [[x, y, z], t, u]=[x, y, [z, t, u]]+\varepsilon(z, t+u)[x, [y, t, u], z]+\varepsilon(y+z, t+u)[[x, t, u], y, z]\label{lci}
\end{eqnarray}
for any $x, y, z, t, u\in \mathcal{H}(A)$.\\
If the trilinear map $[-, -, -]$ is $\varepsilon$-skew-symmetric for any pair of variables, then $(A, [-, -, -], \varepsilon)$ is 
said to be a ternary Lie color algebra \cite{TZ}.
\end{definition}

It is proved in (\cite{JM}, Proposition 3.2) that any Leibniz algebra is also a Leibniz $n$-algebra. We prove the analog for graded case and $n=3$.
That is one can get ternary Leibniz color algebras from Leibniz color algebras.
\begin{theorem}\label{ll3}
Let $(L, [-, -], \varepsilon)$ be a Leibniz color algebra. Then $L$ is a ternary Leibniz color algebra with respect to the bracket
$$[x, y, z]:=[x, [y, z]],$$
for  any $x, y, z\in \mathcal{H}(L)$.
\end{theorem}
\begin{proof}
 Applying twice relation (\ref{cpa}), for any $x, y, z\in \mathcal{H}(L)$, we have
\begin{eqnarray}
 [[x, y, z], t, u]&=&[[x, [y, z]], [t, u]]\nonumber\\
&=&[x,  [[y, z], [t, u]]]+\varepsilon(y+z, t+u)[[x, [t, u]]], [y, z]]\nonumber\\
&=&[x, [y, [z, [t, u]]]]+\varepsilon(z, t+u)[x, [[y, [t, u]], z]]+\varepsilon(y+z, t+u)[[x, [t, u]]], [y, z]]\nonumber\\
&=&[x, y, [z, t, u]]+\varepsilon(z, t+u)[x, [y, t, u], z]+\varepsilon(y+z, t+u)[[x, t, u], y, z].\nonumber
\end{eqnarray}
This completes the proof.
\end{proof}

\begin{proposition}
 Let $(L, [-, -, -], \varepsilon)$ be a ternary Leibniz color algebra, $\xi\in A_0$ such that $[\xi, x, \xi]=0$, for any $x\in L$. Then,
$L$ is a Leibniz color algebra with the bracket
$$\{x, y\}=[x, y, \xi],$$
for any $x, y, z\in \mathcal{H}(L)$.
\end{proposition}
\begin{proof}
 It is straightforward.
\end{proof}

The next assertion connects color trialgebras to ternary Leibniz color trialgebras.
\begin{proposition}
 Let $(A, \dashv, \perp, \vdash, \varepsilon)$ be a color trialgebra. Then $A$ is a ternary Leibniz color algebra with respect to the bracket
\begin{eqnarray}
 [x, y, z]:=x\dashv(y\perp z-\varepsilon(y, z)z\perp y)-\varepsilon(x, y+z)(y\perp z-\varepsilon(y, z)z\perp y)\vdash x,\nonumber
\end{eqnarray}
for  all $x, y, z\in \mathcal{H}(A)$.
\end{proposition}
\begin{proof}
 It follows from a straightforward computation.
\end{proof}

\begin{corollary}
 Let $(A, \cdot, \varepsilon)$ be a non commutative associative color algebra. Then $(A, [-, -, -], \varepsilon)$ is a ternary Leibniz color algebra, where
$$[x, y, z]:=x\cdot(y\cdot z-\varepsilon(y, z)z\cdot y)-\varepsilon(x, y+z)(y\cdot z-\varepsilon(y, z)z\cdot y)\cdot x.$$
\end{corollary}

\begin{remark}
Observe that this ternary Leibniz color algebra is nothing but the one constructed from the commutator of the associative product.
\end{remark}

\begin{proposition}
 Let  $(A, [-, -, -], \varepsilon)$ and $(A',  [-, -, -]', \varepsilon)$ be two  ternary Leibniz color 
 algebras. Then $A\oplus A'$ is also a ternary Leibniz color  algebra with respect to the operations :
\begin{eqnarray}
\{x\oplus x', y\oplus y', z\oplus z'\}&:=&[x, y, z]\oplus[x', y', z']'\nonumber,
\end{eqnarray}
for any $x, y, z\in \mathcal{H}(A)$ and $x', y', z'\in \mathcal{H}(A')$.
\end{proposition}
\begin{proof}
It is straightforward.
\end{proof}

\begin{proposition}
 Let $(A, [-, -, -], \varepsilon)$ be a ternary Leibniz color  algebra. Then $A\otimes A$ also has a Leibniz color algebra structure
 for the structure maps defined by
\begin{eqnarray}
{[x\otimes y, x'\otimes y']}&:=&x\otimes[y, x', y']+\varepsilon(y, x'+y')[x, x', y']\otimes y.\nonumber
\end{eqnarray}
\end{proposition}
\begin{proof}
 It follows from a direct computation.
\end{proof}

\begin{corollary}
 Let $(A, [-, -])$ be a Leibniz color algebra. Then $A\otimes A$ is also a Leibniz color algebra
with respect to the operation
\begin{eqnarray}
{[x\otimes y, x'\otimes y']}:=x\otimes[y, [x', y']]+\varepsilon(y, x'+y')[x, [x', y']]\otimes y.\nonumber
\end{eqnarray}
\end{corollary}

\begin{definition}
 An even linear map $\theta : L\rightarrow L$ is said to be an element of centroid of $L$ if, for any $x, y, z\in\mathcal{H}(L)$,
\begin{eqnarray}
 \theta[x, y, z]=[\theta(x), y, z]=[x, \theta(y), z]=[x, y, \theta(z)].
\end{eqnarray}
\end{definition}

\begin{example}
 Any homothety is an element of centroid of $L$. In deed, putting
$$\theta(x)=\lambda x,\;\; x\in L,\;\; \lambda\in\mathbb{K},$$
we have
\begin{eqnarray}
 \theta[x, y, z]=\lambda[x, y, z]
&=&[\lambda x, y, z]=[\theta(x), y, z]\nonumber\\
&=&[x, \lambda y, z]=[x, \theta(y), z]\nonumber\\
&=&[x, y, \lambda z]=[x, y, \theta(z)].\nonumber
\end{eqnarray}
\end{example}

\begin{proposition}
 Let $L$ be a ternary Leibniz color algebra and $\theta : L\rightarrow L$ an element of centroid of $L$. Then $L$ becomes a ternary Leibniz
 color algebra in each of the following cases :
\begin{itemize}
 \item [a)] $\{x, y, z\}=[\theta(x), y, z]$,
\item [b)] $\{x, y, z\}=[\theta(x), \theta(y), z]$,
\item [c)] $\{x, y, z\}=[\theta(x), \theta(y), \theta(z)]$,
\end{itemize}
 for any $x, y, z\in \mathcal{H}(L)$.
\end{proposition}
\begin{proof}
 It follows from a simple computation.
\end{proof}

\begin{remark}
 Observe that the bracket $c)$ is nothing but the initial one whenever $\theta$ is surjective.
\end{remark}

\begin{definition}
 Given a ternary Leibniz color $L$, an even linear map $P: L\rightarrow L$ is said to be a Reynolds operator on $L$ if one has
\begin{eqnarray}
 [P(x), P(y), P(z)]=P\Big([P(x), P(y), z]+[P(x), y, P(z)]+[x, P(y), P(z)]-[P(x), P(y), P(z)]\Big), \label{ret}
\end{eqnarray}
for any $x, y, z\in\mathcal{H}(L)$.
\end{definition}

\begin{theorem}
 Let $P: L\rightarrow L$ be a Reynolds operator on a ternary Leibniz color algebra $L$. Then, $L$ is also a ternary Leibniz color for the 
product
\begin{eqnarray}
 \{x, y, z\}=[P(x), P(y), z]+[P(x), y, P(z)]+[x, P(y), P(z)]-[P(x), P(y), P(z)],
\end{eqnarray}
for any $x, y, z\in\mathcal{H}(L)$.
\end{theorem}
\begin{proof}
 It comes from direct calculation using relation (\ref{ret}).
\end{proof}

\begin{definition}
 An even linear map $R : L\rightarrow L$ on a ternary Leibniz color algebra is called a Rota-Baxter operator of weight $\lambda\in\mathbb{K}$, if
for any $x, y, z\in\mathcal{H}(L)$,
\begin{eqnarray}
 &&[R(x), R(y), R(z)]=R\Big([R(x), R(y), z]+[R(x), y, R(z)]+[x, R(y), R(z)]\nonumber\\
&&\qquad\qquad\qquad\qquad+\lambda[R(x), y, z]+\lambda[x, R(y), z]+\lambda[x, y, R(z)]+\lambda^2
[x, y, z]\Big).\label{rot}
\end{eqnarray}
\end{definition}

\begin{theorem}
 Given a Rota-Baxter operator $R: L\rightarrow L$ on a ternary Leibniz color algebra $L$, we can make $L$ into another ternary Leibniz color
algebra with the bracket
\begin{eqnarray}
 &&\{x, y, z\}=[R(x), R(y), z]+[R(x), y, R(z)]+[x, R(y), R(z)]\nonumber\\
&&\qquad\qquad+\lambda[R(x), y, z]+\lambda[x, R(y), z]+\lambda[x, y, R(z)]+\lambda^2
[x, y, z],
\end{eqnarray}
for any $x, y, z\in\mathcal{H}(L)$.
\end{theorem}
\begin{proof}
 It follows from direct computation using relation (\ref{rot})
\end{proof}

\begin{theorem}
 Let $A$ be a commutative associative color algebra and $L$ a ternary Leibniz color algebra. Then, for any $a, b, c\in\mathcal{H}(A)$, 
$x, y, z\in\mathcal{H}(L)$,  the bracket
\begin{eqnarray}
 \{a\otimes x, b\otimes y, c\otimes z\}=\varepsilon(x, b+c)\varepsilon(y, c)abc\otimes[x, y, z],
\end{eqnarray}
make $A\otimes L$ into a ternary Leibniz color algebra.
\end{theorem}
\begin{proof}
 It follows from a straightforward computation.
\end{proof}

Now, we introduce bimodule and representation of ternary Leibniz color algebras and establish some properties.

\begin{definition}\label{mde}
 A  color bimodule $M$ over a ternary Leibniz color algebra $(A, \cdot, [-, -, -], \varepsilon)$ is the given of three even trilinear applications
$$[-, -, -] : M\otimes L\otimes L\rightarrow M, \quad [-, -, -] : L\otimes M\otimes L\rightarrow M, \quad
[-, -, -] : L\otimes L\otimes M\rightarrow M$$
satisfying the following sets of axioms
\begin{eqnarray}
 [[m, x, y], z, t]]&=&[m, x, [y, z, t]]+\varepsilon(y, z+t)[m, [x, z, t], y]+\varepsilon(x+y, z+t)[[m, z, t], x, y],\label{lpc3a1}\\
{[[x, m, y], z, t]]}&=&[x, m, [y, z, t]]+\varepsilon(y, z+t)[x, [m, z, t], y]+\varepsilon(m+y, z+t)[[x, z, t], m, y],\\
{[[x, y, m], z, t]]}&=&[x, y, [m, z, t]]+\varepsilon(m, z+t)[x, [y, z, t], m]+\varepsilon(y+m, z+t)[[x, z, t], y, m],\\
{[[x, y, z], m, t]]}&=&[x, y, [z, m, t]]+\varepsilon(z, m+t)[x, [y, m, t], z]+\varepsilon(y+z, m+t)[[x, m, t], y, z],\\
{[[x, y, z], t, m]]}&=&[x, y, [z, t, m]]+\varepsilon(z, t+m)[x, [y, t, m], z]+\varepsilon(y+z, t+m)[[x, t, m], y, z], \label{lpc3a5}
\end{eqnarray}
for all $x, y, z, t\in \mathcal{H}(L), m\in\mathcal{H}(M)$.
\end{definition}

\begin{theorem}
 Let a color bimodule over a ternary Leibniz color algebra $L$. Then $L\oplus M$ is a ternary Leibniz color algebra with respect to the 
multiplication 
\begin{eqnarray}
 \{x_1+m_1, x_2+m_2, x_3+m_3\}=[x_1, x_2, x_3]+[m_1, x_2, x_3]+[x_1, m_2, x_3]+[x_1, x_2, m_3],
\end{eqnarray}
for any $x_i\in\mathcal{H}(L), m_i\in\mathcal{H}(M), i=1, 2, 3$.
\end{theorem}
\begin{proof}
 It follows from direct calculation.
\end{proof}

The proof of the following proposition is easy.
\begin{proposition}
 Let $M$ and $N$ be two color bimodule over the ternary Leibniz color $L$.  Then, $M\oplus N$ is also a color bimodule over $L$ with respect to 
componentwise operation.
\end{proposition}

\begin{theorem}
 Let $M$ and $N$ be two color bimodule over the ternary Leibniz color $L$.  Then, $M\otimes N$ is also a color bimodule over $L$ with respect to the 
structure maps :
\begin{eqnarray}
 \{x, y, m\otimes n\}&=&[x, y, m]\otimes n+\varepsilon(x+y, m)m\otimes [x, y, n],\\
\{x, m\otimes n, y\}&=&\varepsilon(n, y)[x, m, y]\otimes n+\varepsilon(x, m)m\otimes[x, n, y],\\
\{m\otimes n, x, y\}&=&\varepsilon(n, x+y)[m, x, y]\otimes n+m\otimes m\otimes [n, x, y],
\end{eqnarray}
for $x, y\in \mathcal{H}(L), m\in\mathcal{H}(M), n\in\mathcal{H}(N)$.
\end{theorem}
\begin{proof}
 It comes from direct computation using axioms in Definition \ref{mde}.
\end{proof}

\begin{definition}
 Let $L$ be a ternary Leibniz color algebra and $M$ a color vector space. A representation of $L$ on $M$ is the given of three even linear
maps $\lambda : L\times L\rightarrow End(M), (x, y)\mapsto \lambda_{x, y}$, $\mu : L\otimes L\rightarrow End(M), (x, y)\mapsto \mu_{x, y}$ 
and $\rho : L\otimes L\rightarrow End(M), (x, y)\mapsto \rho_{x, y}$ such that :
\begin{eqnarray}
 \lambda_{[x, y, z], t}(m)&=& \lambda_{x, y}\lambda_{z, t}(m)+\varepsilon(z, t+m)\mu_{x, z}\lambda_{y, t}(m)
+\varepsilon(y+z, t+m)\rho_{y, z}\lambda_{x, t}(m),\\
\mu_{[x, y, z], t}(m)&=&\lambda_{x, y}\mu_{z, t}(m)+\varepsilon(z, m+t)\lambda_{x, z}\mu_{y, t}(m)+\varepsilon(y+z, m+t)\rho_{y, z}\mu_{x, t}(m),\\
\lambda_{z, t}\lambda_{x, y}(m)&=&\lambda_{x, y}\lambda_{z, t}(m)+\varepsilon(m, z+t)\lambda_{x, [y, z, t]}(m)
+\varepsilon(y+m, z+t)\lambda_{[x, z, t], y}(m),\\
\rho_{z, t}\mu_{x, y}(m)&=&\mu_{x, [y, z, t]}(m)+\varepsilon(y, z+t)\mu_{x, y}\rho_{z, t}(m)+\varepsilon(y+m, z+t)\mu_{[x, z, t], y}(m),\\
\rho_{z, t}\rho_{x, y}(m)&=&\rho_{x, [y, z, t]}(m)+\varepsilon(y, z+t)\rho_{[x, z, t], y}(m)+\varepsilon(x+y, z+t)\rho_{x, y}\rho_{z, t}(m),
\end{eqnarray}
for any $x, y, z, t\in \mathcal{H}(L), m\in\mathcal{H}(M)$.
\end{definition}

It is well kwnown that every  bimodule $M$ gives rise to a representation $(\lambda, \mu, \rho)$ of $L$ on $M$ via $\lambda_{x, y}(m)=[x, y, m]$,
$\mu_{x, y}(m)=[x, m, y]$ and $\rho_{x, y}(m)=[m, x, y]$.
 Conversely, every representation $(\lambda, \mu, \rho)$ of $L$ on $M$ define an $L$ bimodule structure on $M$ via
$[x, y, m]:=\lambda_{x, y}(m)$,
$[x, m, y]:=\mu_{x, y}(m)$ and $[m, x, y]:=\rho_{x, y}(m)$. Therefore, we have the following proposition.
\begin{proposition}
 If $M$ and $M$ are two representation spaces of a ternary Leibniz color algebra, then $M+N$ and $M\otimes N$ are also representation
spaces of $L$.
\end{proposition}

 \subsection{Ternary Leibniz-Poisson color algebras}
We introduce ternary Leibniz-Poisson color algebras, their bimodules and give some procedures of constructions.

\begin{definition}
 A ternary Leibniz-Poisson color algebra is a quadruple $(A, \cdot, [-, -, -], \varepsilon)$ in which
\begin{enumerate}
 \item [1)] $(A, \cdot, \varepsilon)$ is an associative color algebra,
\item [2)] $(A, [-, -, -], \varepsilon)$ is a ternary Leibniz color algebra,
\item [3)] and the following {\it right ternary Leibniz rule }
\begin{eqnarray}
 [x\cdot y, z, t]=x\cdot [y, z, t]+\varepsilon(y, z+t)[x, z, t]\cdot y \label{cal3} 
\end{eqnarray}
holds for any $x, y, z, t\in \mathcal{H}(A)$.
\end{enumerate}
If in addition, the product $\cdot$ is $\varepsilon$-commutative, then the ternary Leibniz-Poisson color algebra is said to
 be commutative. Moreover, if the trilinear map $[-, -, -]$ is $\varepsilon$-skew-symmetric for any pair of variables, then 
$(A, \cdot, [-, -, -], \varepsilon)$ is called a ternary Poisson color algebra, see \cite{MA} for trivial grading. 
\end{definition}
\begin{example}
 If $(A, \cdot, [-, -, -], \varepsilon)$ is a ternary Leibniz-Poisson color  algebra, then
$(A, \cdot^{op}, [-, -, -], \varepsilon)$ is also a ternary Leibniz-Poisson color  algebra, where $\cdot^{op} : A\otimes A
\rightarrow A, x\otimes y \mapsto\varepsilon(x, y)y\cdot x$.
\end{example}

\subsubsection{Constructions}

The following results affirm that one can get ternary Leibniz-Poisson color algebras from Leibniz-Poisson color algebras.
\begin{theorem}\label{ll}
 Let $(A, \cdot, [-, -], \varepsilon)$ be a Leibniz-Poisson color algebra. Then \\ $(A, \cdot, [-, [-, -]], \varepsilon)$ is a
 ternary Leibniz-Poisson color algebra.
\end{theorem}
\begin{proof}
 It is clear that $(A, \cdot, \varepsilon)$ is an associative color algebra and $(A, [-, [-, -]])$ is a ternary Leibniz color algebra, thanks to
Theorem \ref{ll3}. We now need to prove the compatibility condition (\ref{cal3}). For any $x, y, z, t\in \mathcal{H}(A)$, we have
\begin{eqnarray}
 [x\cdot y, z, t]
&=&[x\cdot y, [z, t]]\nonumber\\
&=&x\cdot [y, [z, t]]+\varepsilon(y, z+t)[x, [z, t]]\cdot y\quad (\mbox{by}\quad (\ref{comp}))\nonumber\\
&=&x\cdot [y, z, t]+\varepsilon(y, z+t)[x, z, t]\cdot y.\nonumber
\end{eqnarray}
This ends the proof.
\end{proof}

\begin{corollary}
Let $(A, \cdot, \varepsilon)$ be an associative color algebra. Define the even bilinear and trilinear maps 
$[-, -] : A^{\otimes2}\rightarrow A$ and $[-, -, -] : A^{\otimes3}\rightarrow A$ respectively by
$$[x, y]:=x\cdot y-\varepsilon(x, y)y\cdot x \quad\mbox{and}\quad [x, y, z]:=[x, [y, z]]$$
Then $(A, \cdot, [-, -, -], \varepsilon)$ is a ternary Leibniz-Poisson color algebra.
\end{corollary}

The following proposition asserts that the direct sum of two  ternary  Leibniz-Poisson color  algebras are also a
 ternary Leibniz-Poisson color  algebra with componentwise operation.
\begin{proposition}
 Let  $(A, \cdot, [-, -, -], \varepsilon)$ and $(A', \cdot', [-, -, -]', \varepsilon)$ be two ternary Leibniz-Poisson color 
 algebras. Then $A\oplus A'$ is a ternary Leibniz-Poisson color  algebra with respect to the operations :
\begin{eqnarray}
 (x\oplus x')\ast(y\oplus y')&:=&(x\cdot y)\oplus (x'\cdot'y')\nonumber\\
\{x\oplus x', y\oplus y', z\oplus z'\}&:=&[x, y, z]\oplus[x', y', z']'\nonumber,
\end{eqnarray}
for any $x, y, z\in \mathcal{H}(A)$ and $x', y', z'\in \mathcal{H}(A')$.
\end{proposition}
\begin{proof}
It is straightforward.
\end{proof}

The below theorem states that the tensor product of ternary Leibniz-Poisson color algebra by itself leads to
 Leibniz-Poisson color algebra.
\begin{theorem}\label{llp}
 Let $(A, \cdot, [-, -, -], \varepsilon)$ be a  ternary Leibniz-Poisson color  algebra. Then $A\otimes A$ is endowed with a 
 Leibniz-Poisson color algebra structure for the structure maps defined by
\begin{eqnarray}
 (x\otimes y)\cdot(x'\otimes y')&:=&\varepsilon(y, x')(x\cdot x')\otimes (y\cdot y')\nonumber\\
{[x\otimes y, x'\otimes y']}&:=&x\otimes[y, x', y']+\varepsilon(y, x'+y')[x, x', y']\otimes y.\nonumber
\end{eqnarray}
\end{theorem}
\begin{proof}
 It follows from a direct computation.
\end{proof}
\begin{corollary}
 Let $(A, \cdot, [-, -])$ be a Leibniz-Poisson color algebra. Then $A\otimes A$ is also a Leibniz-Poisson algebra
with respect to the operations 
\begin{eqnarray}
 (x\otimes y)\cdot(x'\otimes y')&:=&\varepsilon(y, x')(x\cdot x')\otimes (y\cdot y')\nonumber\\
{[x\otimes y, x'\otimes y']}&:=&x\otimes[y, [x', y']]+\varepsilon(y, x'+y')[x, [x', y']]\otimes y.\nonumber
\end{eqnarray}
\end{corollary}
\begin{proof}
 It follows from Theorem \ref{llp} and Theorem \ref{ll}. 
\end{proof}

It is proved in (\cite{JMC}, Lemma 2.8) that every Leibniz-Poisson algebra gives rise to ternary Leibniz algebra. 
In the below theorem, we extend this result the Leibniz-Poisson color algebra case.
\begin{theorem}
 Let $(A, \cdot, [-, -], \varepsilon)$ be a  Leibniz-Poisson color algebra, then\\
$(A, \cdot, \{-, -, -\}, \varepsilon)$ is a  ternary Leibniz-Poisson color algebra, with
$$\{x, y, z\}:=[x, y\cdot z]$$
for  any $x, y, z\in \mathcal{H}(A)$.
\end{theorem}
\begin{proof}
We only need to prove  (\ref{lci}) and {\it right ternary Leibniz} rule  (\ref{cal3}). Then,
 for any $x, y, z, t, u\in \mathcal{H}(A)$, one has :
\begin{eqnarray}
 &&\qquad\{\{x, y, z\}, t, u\}-\{x, y, \{z, t, u\}\}-\varepsilon(z, t+u)\{x, \{y, t, u\}, z\}-\varepsilon(y+z, t+u)\{\{x, t, u\}, y, z\}=\nonumber\\
 &&={[[x, y\cdot z], t\cdot u]]}-[x, y\cdot [z, t\cdot u]]-\varepsilon(z, t+u)[x, [y, t\cdot u]\cdot z]
-\varepsilon(y+z, t+u)[[x, t\cdot u], y\cdot z]\nonumber\\
&&={[x, y\cdot z], t\cdot u]}-[x, y\cdot [z, t\cdot u]-\varepsilon(z, t+u) [y, t\cdot u]\cdot z]
-\varepsilon(y+z, t+u)[[x, t\cdot u], y\cdot z]\nonumber\\
&&={[[x, y\cdot z], t\cdot u]}-[x, [y\cdot z, t\cdot u]-\varepsilon(y+z, t+u)[[x, t\cdot u], y\cdot z].\nonumber
\end{eqnarray}
The last line vanishes thanks to the {\it right Leibniz} identity (\ref{cpa}).
Next,
\begin{eqnarray}
&&\qquad \{x\cdot y, z, t\}-x\cdot \{y, z, t\}-\varepsilon(y, z+t)\{x, z, t\}\cdot y=\nonumber\\
&&={[x\cdot y, z\cdot t]}-x\cdot [y, z\cdot t]+\varepsilon(y, z+t)[x, z\cdot t]\cdot y\nonumber,
\end{eqnarray}
which also vanishes by (\ref{cpa}).
\end{proof}

\subsubsection{Modules over Ternary Leibniz-Poisson color algebras}
In what follows we introduce and give constructions of bimodules over  ternary Leibniz-Poisson color algebras.

\begin{definition}
 A bimodule over a  ternary Leibniz-Poisson color algebra $(A, \cdot, [-, -, -], \varepsilon)$ is a bimodule
 $M$ over the associative color algebra $(A, \cdot, \varepsilon)$ and a bimodule over the ternary Leibniz color 
algebra $L$ such that
\begin{eqnarray}
 [m\cdot x, y, z]=m\cdot [x, y, m]+\varepsilon(x, y+z)[m, y, z]\cdot x,\\
{[x\cdot m, y, z]}=x\cdot [m, y, z]+\varepsilon(m, y+z)[x, y, z]\cdot m,\\
{[x\cdot y, m, z]}=x\cdot [y, m, z]+\varepsilon(y, m+z)[x, m, z]\cdot y,\\
{[x\cdot y, z, m]}=x\cdot [y, z, m]+\varepsilon(y, z+m)[x, z, m]\cdot y,\label{lpc3a7}
\end{eqnarray}
for any $x, y, z, t\in \mathcal{H}(A)$ and $m\in \mathcal{H}(M)$.
\end{definition}

In the below proposition, we construct bimodules on ternary Leibniz color algebras from bimodules over Leibniz color algebras.
\begin{proposition}
 Let $(M, \ast, \ast', \varepsilon)$ be a bimodule over a Leibniz color algebra $(L, [-, -], \varepsilon)$. Define
\begin{eqnarray}
 [x, y, m]:=x\ast (y\ast m),\quad [x, m, y]:=x\ast (m\ast' y), \quad [m, x, y]:=m\ast' [x,  y],
\end{eqnarray}
for all $x, y\in \mathcal{H}(L)$ and $m\in \mathcal{H}(M)$.\\
Then $M$ is a bimodule over the ternary Leibniz color algebra associates to the Leibniz color algebra $L$ (as in theorem \ref{ll3}).
\end{proposition}
\begin{proof}
 It is proved by a straightforward computation.
\end{proof}
In order to have bimodules over ternary Leibniz-Poisson color algebras via morphism, we need the below definition.
\begin{definition}
 Let $(L, \cdot, [-, -, -], \varepsilon)$ and $(L', \cdot', [-, -, -]', \varepsilon)$ be two  ternary Leibniz-Poisson
 color algebras. Let $\alpha : L\rightarrow L'$ be an even linear map such that, for any $x, y, z\in \mathcal{H}(L)$,
$$\alpha(x\cdot y)=\alpha(x)\cdot\alpha(y)\quad\mbox{and}\quad \alpha([x, y, z])=[\alpha(x), \alpha(y), \alpha(z)]'.$$
Then $\alpha$ is called a morphism of  ternary Leibniz-Poisson color algebras.
\end{definition}
Then we have the following result.
\begin{theorem}
 Let $(L, \cdot, [-, -, -], \varepsilon)$ and $(L', \cdot', [-, -, -]', \varepsilon)$ be two ternary Leibniz-Poisson color algebras and 
$\alpha : L\rightarrow L'$ be a morphism of  ternary Leibniz-Poisson color algebras. Define 
\begin{eqnarray}
x\ast m &=&\alpha(x)\cdot' m, \quad m\ast x =m\cdot'\alpha(x), \quad\mbox{and}\quad\label{m1}\\
\{ m, x, y\}&=&[m, \alpha(x), \alpha(y)]',\quad
\{x, m, y\}=[\alpha(x), m, \alpha(y)]',\quad\mbox{and}\quad
\{x, y, m\}=[\alpha(x), \alpha(y), m]',\qquad\label{m2}
\end{eqnarray}
for any $x, y\in \mathcal{H}(L)$ and $m\in \mathcal{H}(L')$. Then, with these five maps, $L'$ is a bimodule over $L$. 
\end{theorem}
\begin{proof}
For all $x, y, z\in \mathcal{H}(L)$, we have
\begin{eqnarray}
 \{\{x, y, z\}, t, m\}&
=&[[\alpha(x), \alpha(y), \alpha(z)]', \alpha(t), m]'\nonumber\\
&=&[\alpha(x), \alpha(y), [\alpha(z), \alpha(t), m]']'+\varepsilon(z, t+m)[\alpha(x), [\alpha(y), \alpha(t), m]', \alpha(z)]'\nonumber\\
&&+\varepsilon(y+z, t+m)[[\alpha(x), \alpha(t), m]', \alpha(y), \alpha(z)]'\nonumber\\
&=&\{x, y, \{z, t, m\}\}+\varepsilon(z, t+m)\{x, \{y, t, m\}, z\}+\varepsilon(y+z, t+m)\{\{x, t, m\}, y, z\}.\nonumber
\end{eqnarray}
The rest of the relations are proved similarly.
\end{proof}
\begin{remark}
 Any  ternary Leibniz-Poisson color algebra is a bimodule over itself. 
\end{remark}

\begin{corollary}
 Let $(L, \cdot, [-, -, -], \varepsilon)$ be a  ternary Leibniz-Poisson color algebra and $\alpha : L\rightarrow L$ be an
 endomorphism of $L$. Then (\ref{m1}) and (\ref{m2}) define another bimodule structure of $L$ over itself. 
\end{corollary}

\subsection{Color Lie triple systems}
\begin{definition}
 A color Lie triple system is a $G$-graded vector space $A$ over a field $\mathbb{K}$ equipped with a bicharacter 
$\varepsilon : G\times G\rightarrow\mathbb{K}^*$ and an even trilinear bracket which satisfies the identity (\ref{lci}), instead of skew-symmetry,
satisfies the conditions
\begin{eqnarray}
 &&\qquad\qquad\qquad[x, y, z]=-\varepsilon(y, z)[x, z, y], \quad(\mbox{\it right}\; \varepsilon\mbox{\it -skew-symmetry})\label{rss}\\
&&\varepsilon(z, x)[x, y, z]+\varepsilon(x, y)[y, z, x]+\varepsilon(y, z)[z, x, y]=0,
\;\;(\mbox{\it ternary}\;\;\varepsilon\mbox{\it -Jacobi identity} )\label{ltsji}
\end{eqnarray}
for each $x, y, z\in\mathcal{H}(A)$.
\end{definition}
\begin{theorem}
 Let $(A, [-, -], \varepsilon)$ be a Lie color algebra. Define the even trilinear map $[-, -, -] : A\otimes A\otimes A\rightarrow A$ by
$$[x, y, z]:=[x, [y, z]]$$
for any $x, y, z\in\mathcal{H}(A)$.
Then $(A, [-, -, -], \varepsilon)$ is a color Lie triple system.
\end{theorem}
\begin{proof}
The right skew-symmetry (\ref{rss}) is immediate. The ternary $\varepsilon$-Jacobi identity (\ref{ltsji}) follows from $\varepsilon$-Jacobi
 identity. And Theorem \ref{ll3} completes the proof.
\end{proof}
\begin{corollary}
 Let $(A, \cdot, \varepsilon)$ be an associative color algebra. Then $A$ is a color Lie triple system with respect to the bracket
$$[x, y, z]:=x\cdot (y\cdot z)-\varepsilon(y, z)x\cdot (z\cdot y)
-\varepsilon(x, y+z)(y\cdot z)\cdot x+ \varepsilon (x, y+z)\varepsilon(y, z)(z\cdot y)\cdot x.$$
\end{corollary}
Next we introduce color Jordan triple systems and study their connection with color Lie triple systems.
\begin{definition}
 A color Jordan triple system is a triple $(J, [-, -, -], \varepsilon)$ in which
 $J$ is a $G$-graded vector space over a field $\mathbb{K}$, $\varepsilon : G\times G\rightarrow\mathbb{K}^*$ is a bicharacter and 
$[-, -, -] : A\otimes A\otimes A\rightarrow A$ is an even trilinear map satisfying
\begin{eqnarray}
 [x, y, z]&=&\varepsilon(x, y)\varepsilon(x, z)\varepsilon(y, z)[z, y, x]\quad\mbox{\it (outer-}\varepsilon\mbox{\it -symmetry}) \label{oss}
\end{eqnarray}
and {\it color Jordan triple identity}
\begin{eqnarray}
{[[x, y, z], t, u]}&=&[x, y, [z, t, u]]-\varepsilon(z, t+u)\varepsilon(t, u)[x, [y, u, t], z]+\varepsilon(y+z, t+u)[[x, t, u], y, z],
\label{jts}
\end{eqnarray}
for any $x, y, z\in\mathcal{H}(J)$.
\end{definition}
\begin{example}
 Let $(A, \cdot, \varepsilon)$ be an associative color algebra. Then $(A, [-, -, -], \varepsilon)$ is a color Jordan triple system with respect
to the triple product
$$[x, y, z]:=x\cdot y\cdot z+\varepsilon(x, y)\varepsilon(x, z)\varepsilon(y, z)z\cdot y\cdot x.$$
\end{example}
\begin{example}
Let $(A, \cdot, \varepsilon)$ be an associative color algebra and $\theta : A\rightarrow A$ be an even linear map on $A$ satisfying $\theta^2=Id_A$ and 
$\theta(x\cdot y)=\varepsilon(x, y)\theta(y)\cdot\theta(x)$ for any $x, y\in \mathcal{H}(A)$. Then $(A, [-, -, -], \varepsilon)$ is a color Jordan
triple system with the triple product
$$[x, y, z]:=x\cdot\theta(y)\cdot z+\varepsilon(x, y)\varepsilon(x, z)\varepsilon(y, z)z\cdot\theta(y)\cdot x.$$
\end{example}
We have the following result.
\begin{theorem}
 Let $(J, [-, -, -], \varepsilon)$ be a color Jordan triple system. Define the triple product 
$$\{x, y, z\}:=[x, y, z]-\varepsilon(y, z)[x, z, y]$$
for any $x, y, z\in\mathcal{H}(J)$. Then $L(J)=(J, \{-, -, -\}, \varepsilon)$ is a color Lie triple system.
\end{theorem}
\begin{proof}
The {\it right $\varepsilon$-skew-symmetry} is immediate. The {\it ternary $\varepsilon$-Jacobi identity} follows from (\ref{oss}). 
It remains to check the  identity (\ref{lci}) for $\{-, -, -\}$. For any $x, y, z\in\mathcal{H}(J)$, we have
\begin{eqnarray}
 \{\{x, y, z\}, t, u\}
&=&[\{x, y, z\}, t, u]-\varepsilon(t, u)[\{x, y, z\}, u, t]\nonumber\\
&=&[([x, y, z]-\varepsilon(y, z)[x, z, y]), t, u]-\varepsilon(t, u)[([x, y, z]-\varepsilon(y, z)[x, z, y]), u, t]\nonumber\\
&=&[[x, y, z], t, u]-\varepsilon(y, z)[[x, z, y], t, u]-\varepsilon(t, u)[[x, y, z], u, t]
+\varepsilon(t, u)\varepsilon(y, z)[[x, z, y], u, t].\nonumber
\end{eqnarray}
Similarly,
\begin{eqnarray}
 \{x, y, \{z, t, u\}\}&=& [x, y, [z, t, u]]-\varepsilon(t, u)[x, y, [z, u, t]]-\varepsilon(y, z+t+u)[x, [z, t, u], y]
+\varepsilon(y, z+t+u)\varepsilon(t, u),\nonumber\\
\{x, \{y, t, u\}, z\}&=&[x, [y, t, u], z]-\varepsilon(t, u)[x, [y, u, t], z]-\varepsilon(y+t+u, z)[x, z, [y, t, u]]\nonumber\\
&&+\varepsilon(y+t+u, z)\varepsilon(t, u)[x, z, [y, u, t]],\nonumber\\
\{\{x, t, u\}, t, u\}&=&[[x, t, u], y, z]-\varepsilon(t, u)[[x, u, t], y, z]-\varepsilon(y, z)[[x, t, u], z, y]
+\varepsilon(y, z)\varepsilon(t, u)[[x, u, t], z, y]\nonumber.
\end{eqnarray}
Then, using axiom (\ref{jts}), relation (\ref{lci}) holds for the bracket $\{-, -, -\}$.
\end{proof}
\subsection{Comstrans color algebras}
\begin{definition}
 A Comstrans algebra is a $G$-graded vector space $T$ with a bicharacter $\varepsilon : G\times G\rightarrow \mathbb{K}^*$,
 two trilinear operations $A\times A\times A\rightarrow A$, the commutator $(x, y, z)\mapsto [x, y, z]$ and the translator 
$(x, y, z)\mapsto < x, y, z>$ satisfying the following identities for all $x, y, z\in\mathcal{H}(T)$,
\begin{eqnarray}
<x, y, x>=[x, y, x],\qquad\qquad\qquad\qquad\quad\\
 {[x, y, z]=-\varepsilon(y, z)[x, z, y]},\qquad\qquad\qquad\qquad\\
\varepsilon(z, x)<x, y, z>+\varepsilon(x, y)<y, z, x>+\varepsilon(y, z)<z, x, y>=0.
\end{eqnarray}
\end{definition}

\begin{theorem}
 Let $(A, \cdot, \varepsilon)$ be an associative color. Then, $A$ carries a structure of  Comstrans color algebra with the multiplications
\begin{eqnarray}
 [x, y, z]&=&x\cdot y\cdot z-\varepsilon(y, z)x\cdot y\cdot z,\\
<x, y, z>&=&x\cdot y\cdot z-\varepsilon(x+y, z)z\cdot x\cdot y,
\end{eqnarray}
$x, y, z\in\mathcal{H}(A)$.
\end{theorem}

\begin{proposition}
 Any Lie color algebra $(L, [-, -], \varepsilon)$ has an underlying Comstrans color algebra defined by
\begin{eqnarray}
 [x, y, z]=<x, y, z>=[x, [y, z]]
\end{eqnarray}
for any $x, y, z\in\mathcal{H}(L)$.
\end{proposition}

\begin{proposition}
 Let $T$ be a $G$-graded vector space, $\varepsilon$ a bicharacter on $G$ and $f : T\times T\rightarrow\mathbb{K}$ a $\varepsilon$-symmetric 
bilinear form on $T$. Then, $T$ becomes  a Comstrans color algebra with the operations
\begin{eqnarray}
 [x, y, z]=<x, y, z>=f(x, z)y-\varepsilon(x, z)f(y, z)x
\end{eqnarray}
\end{proposition}

{\bf Acknowlegments :} 
The author would like to thank Professor Alain Togb\'e of Purdue University for his material support.


\label{lastpage-01}
\end{document}